\documentclass[11pt,leqno]{amsart}

\usepackage{anysize}
\marginsize{3.5cm}{3.5cm}{3cm}{3cm}

\usepackage[]{hyperref}
\hypersetup{
    colorlinks=true,       
    linkcolor=red,          
    citecolor=blue,        
    filecolor=magenta,      
    urlcolor=cyan           
}

\usepackage{amsmath}
\usepackage{amsfonts,amssymb}
\usepackage{enumerate}
\usepackage{mathrsfs}
\usepackage{amsthm}
\usepackage{tikz}

\theoremstyle{plain}
\newtheorem{theo}{Theorem} 

\newtheorem{lemm}[theo]{Lemma}

\theoremstyle{definition}
\newtheorem{rema}[theo]{Remark}
\newtheorem{remas}[theo]{Remarks}

\newtheorem{exs}[theo]{Examples}

\DeclareSymbolFont{pletters}{OT1}{cmr}{m}{sl}
\DeclareMathSymbol{s}{\mathalpha}{pletters}{`s}
\def\ba{\begin{align}}
\def\bad{\begin{aligned}}
\def\be{\begin{equation}}
\def\ea{\end{align}}
\def\ead{\end{aligned}}
\def\ee{\end{equation}}
\def\e{\eqref}

\def\eps{\varepsilon}

\def\la{\left\vert}

\def\le{\leq}

\def\mez{\frac{1}{2}}

\def\ra{\right\vert}

\def\xN{\mathbf{N}}
\def\xR{\mathbf{R}}

\title{A    stationary  phase type estimate }
\author{
T. Alazard} 
\address{D\'epartement de Math\'ematiques, UMR 8553 du CNRS, Ecole Normale Sup\'erieure, 45, rue d'Ulm 75005 Paris Cedex, France}
\author{N. Burq}
\address{Laboratoire de Math\'ematiques d'Orsay, UMR 8628 du CNRS, Universit\'e Paris-Sud, 91405 Orsay Cedex, France}
\author{C. Zuily}
\address{Laboratoire de Math\'ematiques d'Orsay, UMR 8628 du CNRS, Universit\'e Paris-Sud, 91405 Orsay Cedex, France}

\thanks{T.A., N.B. and C.Z. were supported in part by Agence Nationale de la Recherche
  project ANA\'E ANR-13-BS01-0010-03. N.B. was supported in part by Agence Nationale de la Recherche project  NOSEVOL, 2011 BS01019 01.}

\begin{document}
 \begin{abstract}The purpose of this note is to prove a stationary phase estimate well adapted to parameter dependent phases. In particular, no discussion is made on the positions (and behaviour) of critical points, no lower or upper bound on the gradient of the phase is assumed, and the dependence of the constants with respect to derivatives of the phase and symbols is explicit.
 \end{abstract}
\maketitle

For a fixed phase, the stationnary phase lemma (and its simplified version, the stationary phase estimate) is a very well understood tool which provides very good estimates for oscillatory integrals of the type 
\begin{equation}\label{eq.statio}
 I_{\phi, b}(\lambda)= \int_{\mathbb{R}^d} e^{i\lambda \Phi(\xi)} b(\xi) d\xi \Rightarrow |I_{\Phi, b}( \lambda)| \leq C \lambda^{- \frac d 2}
\end{equation}
The method of proof is quite standard and follows the classical path: 
\begin{enumerate}
\item Using the non degeneracy of the hessian of the phase, one knows that the critical points are isolated, hence for a compactly supported symbol there are finitely many such critical points.
\item Away from the critical points, the non stationary estimates (obtained for example by integrating by parts $N$ times with the operator 
$$L= \frac{ \nabla_\xi \Phi  \cdot \nabla_\xi} { i\lambda |\nabla_\xi \Phi|^2 }, $$ gives an estimate bounded by $C_N \lambda^{-N}$
\item Near each critical point, performing first a change of variables (the Morse Lemma) to reduce to the case where the phase is quadratic, and then and exact calculation in Fourier variables gives the estimate~\eqref{eq.statio}
\end{enumerate}
When   $d=1$, Van der Corput Lemma provides a very robust estimate. 

However, in higher dimensions, the situation is less simple, in particular when considering parameter dependent phases (with parameters living in a non-compact domain), where 
\begin{enumerate}
\item even away from the critical points, $\nabla_\xi \Phi$ can degenerate,
\item the determinant of the hessian can degenerate, 
\item the number of critical points can blow-up.
\end{enumerate}
In view of numerous applications (for example dispersion estimates for solutions to PDE's), a precise control of the behaviour, with respect to the phase and symbol,  of the constant $C$ in~\eqref{eq.statio} is necessary. Many robust methods to prove~\eqref{eq.statio} have been developped (see for example~\cite{BDC, MS, FRT, S}). However, it seems that none of these results gives an estimate directly applicable to general situations. This was the motivation for this note.
Let
$$I(\lambda) = \int_{\xR^d} e^{i \lambda \Phi(\xi)} b(\xi)\, d\xi $$
where $\Phi \in C^\infty(\xR^d)$ is a real phase, $b \in C_0^\infty(\xR^d)$ is a symbol. 

We shall set $K = \text{supp }b$  and let $V$ be a small open neighborhood of $K$.
We shall assume that    
\begin{equation}\label{hypE}
\begin{aligned}
(i)&\quad \mathcal{M}_k:= \sum_{2 \leq \vert \alpha \vert \leq k}\text{sup}_{\xi \in V} \vert D^\alpha_\xi \Phi(\xi) \vert  <+ \infty ,\quad   2 \leq k \leq d+2,\\
(ii)&\quad  \mathcal{N}_l:= \quad  \sum_{ \vert \alpha \vert \leq l} \text{sup}_{\xi \in K} \vert D^\alpha_\xi b(\xi) \vert  <+ \infty ,\quad l \leq d+1,\\
  (iii)& \quad \vert \text{det } \text{Hess } \Phi(\xi) \vert \geq a_0>0, \forall \xi \in V, \\
  \end{aligned}
 \end{equation}
 where $\text{Hess } \Phi$ denotes the Hessian matrix of $\Phi$. 
 
\begin{theo}\label{PS0}
There exists a constant $C$ such that, for all $(\Phi,b)$ satisfying   assumptions \eqref{hypE}, and for all $\lambda\geq 1$, 
\begin{equation}\label{estimation0}
\vert I(\lambda) \vert \leq  \frac{C}{a_0^{1+d}}  \big(1+\mathcal{M}_{d+2}^{\frac{d}{2}+d^2}\big)\mathcal{N}_{d+1}\lambda^{-\frac{d}{2}}.
\end{equation}
\end{theo}
\begin{remas}
 
 1.  We notice that no upper bound (nor lower bound) on $\nabla \Phi$ is required. This is important in particular in the case were  the phase $\Phi$ depends on parameters. For instance, in some cases the phase $\Phi$  is of the form $\Phi(x,y, \xi) = (x-y)\cdot \xi + \phi(x,y,\xi)$ where $x,y $ are in $\xR^d$. In these case $\nabla \Phi =x-y + \nabla \phi$ and there is no natural upper nor lower bound for it.

   2.  Here is another example (see \cite{ABZ}). Assume  $\Phi(x,y, \xi) = (x-y)\cdot \xi + t \theta (x,y,\xi)$ where $t\in (0,T) $ and $x,y $ are in $\xR^d.$ Assume that $\Phi$ and $b$ satisfiy $(i), (ii)$ uniformly in $(t,x,y)$ and that $\vert \text{det Hess } \theta \vert \geq c>0 $ where $c$ depends only on the dimension $d$. Then setting $X = \frac{x}{t}, Y = \frac{y}{t}$ we   write $i\lambda \Phi = i \lambda t \big( (X-Y)\cdot \xi + \theta(t, tX,tY, \xi) \big)$ and we may apply Proposition \ref{PS0} with $a_0 =c$ and $\lambda$    replaced by $\lambda t.$ We obtain an estimate of $I(\lambda)$ by $t^{-\frac{d}{2}} \lambda ^{-\frac{d}{2}}$ as soon as $t \geq \lambda^{-1}.$
 
  3. The term $a_0^{1+d}$ in \eqref{estimation0} can be written as $a_0 a_0^d$ where the factor $a_0^d$ 
comes from the possible occurence of $a_0^{-d}$ critical points of the phase on the support of $b$. In the case where $\Phi$ has only one non degenerate critical point this term could be avoided. In this direction  we have the following result.
 \end{remas}
\begin{theo}\label{PS1}
Assume that $\Phi$ and $b$ satisfy the assumptions \eqref{hypE} and  that  the map  
\begin{equation}\label{inject}
 V  \to \xR^d, \quad  \xi \mapsto \nabla \Phi(\xi),\quad \text{ is injective}.
 \end{equation}

Then one can find   $C>0$ depending only on the dimension $d$ such that
\begin{equation}\label{estimation1}
\vert I(\lambda) \vert \leq  \frac{C}{a_0 }  \big(1+\mathcal{M}_{d+2}^{\frac{d}{2}}\big)\mathcal{N}_{d+1}\lambda^{-\frac{d}{2}}.
\end{equation}
 \end{theo}
 Here are  two examples where Theorem \ref{PS1} appplies.
 \begin{exs}
 1. Assume besides \eqref{hypE} that
 \begin{equation}\label{defpos}
  \langle \text{Hess }\Phi(\xi)X,X\rangle >0, \quad \forall \xi \in V, \quad  \forall X \in \xR^d.
  \end{equation}
 then \eqref{inject} is satisfied.

 For simplicity we shall assume that the neighborhood $V$ of $\text{supp }b$ appearing in \eqref{hypE} $(i)$ is convex. First of all, since the symmetric matrix $\text{Hess } \Phi $ is a non negative,  its eigenvalues are non negative. It follows from the hypothesis  \eqref{hypE} $(iii)$ (see \eqref{vp>}) that 
  \begin{equation}\label{defpos}
   \langle \text{Hess } \Phi(\xi) X,X \rangle \geq \frac{a_0}{(C_dM_2)^{d-1}}\vert X \vert^2,\quad \forall \xi \in V, \quad \forall X \in \xR^d.
   \end{equation}
With $\xi, \eta \in V$  we   write
\begin{align*}
 \nabla \Phi(\xi) - \nabla \Phi(\eta) &= \int_0^1\frac{d}{ds}\big[ \nabla \Phi (s \xi + (1-s)\eta)\big] \, ds,\\
 &= \int_0^1\text{Hess }\Phi(s \xi + (1-s)\eta) \cdot (\xi-\eta)\, ds.
 \end{align*}
 It follows from   \eqref{defpos} that
 $$\langle \nabla \Phi(\xi) - \nabla \Phi(\eta), \xi-\eta \rangle \geq  \frac{a_0}{(C_dM_2)^{d-1}} \vert \xi -\eta \vert^2 $$
from which we deduce that
$$ \frac{a_0}{(C_dM_2)^{d-1}}  \vert \xi -\eta \vert \leq \vert  \nabla \Phi(\xi) - \nabla \Phi(\eta) \vert $$
which completes the proof.

2. Let  $A$ be a real, symmetric, non singular $d \times d$ matrix and $\Psi$ be a smooth phase such that $\mathcal{M}_{d+2} (\Psi)<+\infty$. Set $\Phi(\xi) = \mez \langle A\xi,\xi \rangle + \eps \Psi(\xi).$ Then if $\eps$ is small enough  the assumptions in Theorem \ref{PS1} are satisfied.   \end{exs}
\begin{rema}
Notice that the estimates \eqref{estimation0}, \eqref{estimation1} do not seem to be optimal with respect to the power of $a_0$ since according to the usual stationnary phase method one could expect to have  $a_0^{-\mez}$ in the right hand side.
\end{rema}

Actually  it is sufficient to prove the following weaker inequality.

\begin{theo}\label{PS2}
1. Under the hypotheses of Theorem \ref{PS0} there exists $\mathcal{F}:\xR^+ \to \xR^+$ non decreasing   such that for every $\lambda\geq 1$
\begin{equation}\label{estimation2}
 \vert I(\lambda) \vert \leq \mathcal{F}(\mathcal{M}_{d+2}) \mathcal{N}_{d+1} \frac{1}{a_0^{1+d}}\lambda^{-\frac{d}{2}}.
 \end{equation}
 2. Under the hypotheses of Theorem \ref{PS1}
 there exists $\mathcal{F}:\xR^+ \to \xR^+$ non decreasing   such that for every $\lambda\geq 1$
\begin{equation}\label{estimation3}
 \vert I(\lambda) \vert \leq \mathcal{F}(\mathcal{M}_{d+2}) \mathcal{N}_{d+1} \frac{1}{a_0}\lambda^{-\frac{d}{2}}.
 \end{equation}

\end{theo}

\begin{proof}[Proof of Theorems \ref{PS0}  and \ref{PS1} given Theorem~\ref{PS2}]

We assume that \eqref{estimation1} is proved and our goal is to deduce that \e{estimation0} holds with 
$C=\mathcal{F}(1)$.  Set $t=  1+ \mathcal{M}_{d+2}  $ and consider $\lambda\ge 1$. Since $t\lambda\ge 1$ we can apply 
 \eqref{estimation2} with $(\lambda,\Phi)$ replaced with $(t\lambda,\Phi(\xi)/t)$ to deduce that
\begin{align*}
\la\int_{\xR^d} e^{i \lambda \Phi(\xi)} b(\xi)\, d\xi\ra
&=\la \int_{\xR^d} e^{i t\lambda \frac{\Phi(\xi)}{t}} b(\xi)\, d\xi\ra\\
&\le 
\mathcal{F}\left(\mathcal{M}_{d+2}\left(\frac{\Phi}{t}\right)\right) \mathcal{N}_{d+1} \frac{1}{(a_0/t^d)^{1+d}}(t\lambda)^{-\frac{d}{2}}\\
&\le \mathcal{F}\left(\frac{\mathcal{M}_{d+2}(\Phi)}{t}\right) \mathcal{N}_{d+1} \frac{1}{a_0^{1+d}}\lambda^{-\frac{d}{2}}t^{\frac{d}{2}+d^2},
\end{align*}
which yields the wanted estimate. The case $2.$ is analogue.
\end{proof}
 
We are left with the proof of Theorem \ref{PS2}. We begin by some preliminaries.
 \subsection{Preliminaries}
  In that follows we shall denote by $C_d$ a positive constant depending only on the dimension $d$ and by $\mathcal{F}$ a non decreasing function from $\xR^+$ to $\xR^+$ which  can change from line to line.

Point 1. First of all we may assume that
\begin{equation}\label{reduc1}
\lambda^\mez a_0 \geq 1.
\end{equation}
Indeed if $\lambda^\mez a_0 \leq 1$ then $1   \leq (\lambda^\mez a_0)^{-d} $ and we write 
$$\vert I(\lambda)\vert \leq \Vert b \Vert_{L^1(\xR^d)} \leq \Vert b \Vert_{L^1(\xR^d)} \frac{a_0}{a_0}\frac{1}{(\lambda^\mez a_0)^{ d}}  \leq  C \mathcal{M}_2^d\Vert b \Vert_{L^1(\xR^d)} \frac{1}{a_0^{1+d}}\lambda^{-\frac{d}{2}} $$
since by \eqref{hypE}, $(iii)$ we have $a_0 \leq C \mathcal{M}_2^d.$

 Point 2.   Set $H = \text{Hess} \Phi.$ By $(i)$ the eingenvalues $(\lambda_j)_{j= 1, \ldots,d}$ of $H$ are bounded by $C_d\mathcal{M}_2$. It follows from   $(iii)$   that  
 \begin{equation}\label{vp>}
  \vert \lambda_j \vert \geq   \frac{a_0}{(C_d \mathcal{M}_2)^{d-1}}, \quad 1 \leq j \leq d. 
  \end{equation}  Therefore
        \begin{equation}\label{H} 
   \vert H(\xi) X \vert \geq  \frac{a_0}{(C_d \mathcal{M}_2)^{d-1}} \vert X \vert, \quad\forall \xi \in V,\quad  \forall X \in \xR^d.
   \end{equation}
   We shall use the Taylor formula
   \begin{equation}\label{Taylor}
     \nabla \Phi(\xi)  = \nabla \Phi(\eta) + \text{Hess }\Phi(\eta)(\xi-\eta) + R, \quad 
   \vert R \vert  \leq C'_d\mathcal{M}_3 \vert \xi - \eta \vert^2.
  \end{equation}
 \begin{lemm}\label{nablaPhi-inj}
   Let $\delta>0$ be defined by
 \begin{equation}\label{delta}
\delta = \frac{a_0}{12C'_d \mathcal{M}_3(C_d\mathcal{M}_2)^{d-1}}.
\end{equation}
  where $C_d, C'_d$ have been defined above.   Let $\xi^* \in \text{supp }b$  
 Then, 
\begin{equation}\label{injective}
\text{ on the ball } B(\xi^*, \delta) \text{ the map }\xi \mapsto \nabla \Phi(\xi) \text{ is injective}.
\end{equation}
\end{lemm}
\begin{proof}
Indeed by \eqref{Taylor} and \eqref{H} if $\xi, \eta \in B(\xi^*, \delta)$ we can write
$$\vert \nabla \Phi(\xi) - \nabla \Phi(\eta)\vert \geq  \Big(\frac{a_0}{(C_d \mathcal{M}_2)^{d-1}}  -  \frac{ a_0}{6  (C_d\mathcal{M}_2)^{d-1}}\Big)\vert \xi - \eta \vert \geq \frac{5 a_0}{6(C_d \mathcal{M}_2)^{d-1}}\vert \xi -\eta \vert .$$
\end{proof}
  Let  $(\xi_j^*)_{ j = 1, \ldots,J} \subset  \text{supp }b,$ such that $\text{supp }b \subset \cup_{j=1}^J B(\xi_j^*, \delta).$ Taking a partition of unity $(\chi_j)$  and setting $b_j = \chi_j b$ we have 
 \begin{equation}\label{Ij0}
 I(\lambda) = \sum_{j=1}^J I_j(\lambda), \quad I_j(\lambda) = \int_{\xR^d} e^{i \lambda \Phi(\xi)} b_j(\xi)\, d\xi.
 \end{equation}
Notice that $\chi_j$ can be taken of the form $\chi_0\big(\frac{\xi-\xi_j^*}{\delta}\big)$ so that
\begin{equation}\label{derchij}
\vert \partial_\xi^\alpha \chi_j(\xi) \vert \leq C \delta^{-\vert \alpha \vert}.
\end{equation}

 Notice also that $J  \leq C_d\delta^{-d}.$  
  \begin{lemm}
  Let $i\in \{1, \ldots,d\}$ and $A_i = \frac{\partial_i \Phi}{\vert \nabla \Phi \vert^2}.$ One can find   $\mathcal{F}: \xR^+ \to \xR^+$ non decreasing such that
\begin{equation}\label{est:Ai}
\vert D_\xi^\alpha A_i (\xi)\vert \leq \mathcal{F}(M_{1+ \vert \alpha \vert}) \sum_{k=2}^{1+ \vert \alpha \vert} \frac{1}{\vert \nabla \Phi(\xi)\vert^k}, \quad \vert \alpha \vert \geq 1.
\end{equation}
\end{lemm}
\begin{proof}
We proceed by induction on $\vert \alpha \vert.$ A simple computation shows that \eqref{est:Ai} is true for $\vert \alpha \vert =1.$  Assume it is true for $\vert \alpha \vert \leq l$ and let $\vert \gamma \vert = l+1 \geq 2.$ Differentiating $\vert \gamma \vert$ times the equality $\vert \nabla \Phi \vert^2 A_i = \partial_i \Phi $ we obtain
\begin{equation}\label{rec1}
\begin{aligned}
  &\vert \nabla \Phi \vert^2 D_\xi^\gamma A_i  = (1) -(2)- (3), \quad  \text{with} \quad   (1)  = D_\xi^\gamma \partial_i \Phi   \\
   &(2)  = \sum_{\vert \beta\vert  =1} \binom{\gamma}{\beta} (D_\xi^\beta\vert \nabla \Phi \vert^2) D^{\gamma - \beta }A_i, \quad 
  (3)  = \sum_{2 \leq \vert \beta \vert \leq \vert \gamma \vert} \binom{\gamma}{\beta} (D_\xi^\beta\vert \nabla \Phi \vert^2) D^{\gamma - \beta }A_i.
  \end{aligned}
   \end{equation}
   We have $\vert (1) \vert \leq \mathcal{M}_{l+2}.$ By the induction, $\vert (2)\vert \leq C \mathcal{M}_2\vert \nabla \Phi \vert  \mathcal{F}(M_{l+1})\sum_{k=2}^{l+1} \frac{1}{\vert \nabla \Phi \vert^k}.$ Now for $ \vert \beta\vert \geq 2$ we have $\vert \gamma\vert-\vert \beta \vert \leq l-1.$ Since $\vert   D_\xi^\beta\vert \nabla \Phi \vert^2 \vert \leq C_2 \mathcal{M}_{ \vert \beta\vert +1} \vert \nabla \Phi \vert + \mathcal{F}( \mathcal{M}_{\vert \beta \vert})$ the induction shows that $\vert(3)\vert \leq C \mathcal{F}(\mathcal{M}_{l+2})(1+ \vert \nabla \Phi \vert)\sum_{k=2}^{l } \frac{1}{\vert \nabla \Phi \vert^k}.$ Dividing both members of the first equation in \eqref{rec1} by $\vert \nabla \Phi \vert^2$ we obtain eventually
   $$\vert  D_\xi^\gamma A_i  \vert \leq \mathcal{F}(\mathcal{M}_{l+2}) \sum_{k=2}^{l+2 } \frac{1}{\vert \nabla \Phi \vert^k}, \quad \vert \gamma \vert = l+1.$$ This completes the proof of \eqref{est:Ai}.
     \end{proof}
  We shall also need the following result. 
  \begin{lemm}
  On the set $\{\xi: 0< \vert \nabla \Phi(\xi) \vert \leq 2\}$ we have
  \begin{equation}\label{est:nablaphi}
  \vert \partial_\xi^\alpha \vert \nabla \Phi(\xi) \vert\vert  \leq \mathcal{F}(\mathcal{M}_{\vert \alpha \vert +1})\vert \nabla \Phi(\xi) \vert^{1- \vert \alpha \vert}, \quad \vert \alpha \vert \geq 1.
\end{equation}
\end{lemm}
\begin{proof}
The proof goes by induction on $\vert \alpha \vert$. For $\vert \alpha \vert =1$ it is a simple computation. Assume this is true for  $1\leq \vert \alpha \vert \leq k$ let  $\vert \gamma \vert = k+1\geq 2.$ Set $F(\xi) =  \vert \nabla \Phi(\xi) \vert.$ Then we write 
$$\partial_\xi^\gamma \big(F(\xi) F(\xi)\big) \vert = \partial_\xi^\gamma \sum_{j=1}^d (\partial_j \Phi)^2.$$
 The right hand side is bounded by $\mathcal{F}(\mathcal{M}_{\vert \gamma \vert +1})(1+ F(\xi)).$ By the Leibniz formula the left hand side can be written as $2 F(\xi)  \partial_\xi^\gamma F(\xi)  $ plus a finite sum of terms of the form $\big(\partial_\xi^{\gamma_1} F(\xi)   \big)\big(\partial_\xi^{\gamma_2} F(\xi)\big)$ where $1 \leq \vert \gamma_j \vert \leq k$ and $\gamma_1 + \gamma_2= \gamma.$ For these last terms we can use the induction and we obtain
 $$F(\xi) \partial_\xi^\gamma F(\xi) \leq \mathcal{F}(\mathcal{M}_{\vert \gamma \vert +1})(1+F(\xi)+  F(\xi) \vert^{2- \vert \gamma \vert}). $$
 Dividing both members by $F(\xi)$ we obtain
 \begin{align*}
  \vert \partial_\xi^\gamma F(\xi)\vert  &\leq \mathcal{F}(\mathcal{M}_{\vert \gamma \vert +1})\Big(1+ \frac{1}{F(\xi)}+ F(\xi)^{1- \vert \gamma \vert}\Big)\\ 
  & \leq  \mathcal{F}(\mathcal{M}_{\vert \gamma \vert +1})F(\xi)^{1- \vert \gamma \vert} \big(F(\xi)^{  \vert \gamma \vert-1} +F(\xi)^{  \vert \gamma \vert-2} +1\big).
 \end{align*}
 Since $\vert \gamma \vert \geq 2$ and $F(\xi) \leq 2$ we obtain the desired result.
 \end{proof}
 \begin{lemm}
  Consider the vector field $\mathcal{L} = A\cdot \nabla$ where  $A_i = \frac{\partial_i \Phi}{\vert \nabla \Phi \vert^2}.$ For any $N \in \xN$ we have
   \begin{equation}\label{Ltranspose}
     (\,  {}\!^t \mathcal{L})^N  = \sum_{\vert \alpha \vert \leq N} c_{\alpha,N}\partial^\alpha,  \quad \text{with} \quad  
   \vert \partial_\xi^\beta c_{\alpha,N} \vert  \leq \mathcal{F}(\mathcal{M}_{ N- \vert \alpha \vert + \vert\beta \vert +1})\sum_{k=N}^{2N-\vert \alpha\vert  + \vert \beta \vert } \frac{1}{\vert \nabla \Phi \vert^k}.
  \end{equation}
  (Here  we set $\mathcal{M}_1 = 1.$ It occurs when $\beta = 0, \vert \alpha\vert =N.$)
  \end{lemm}
  \begin{proof}
Again we proceed by  induction on $N.$ For $N=1$ we have $c_{\alpha,N} = A_i$ if $\vert \alpha \vert =1$ and $c_{\alpha,N} = \text{div} A$ if $\vert \alpha \vert =0.$ Then \eqref{Ltranspose} follows immediately from \eqref{est:Ai}. Assume that \eqref{Ltranspose} is true up to the order $N$ and let us prove it  for $N+1.$ We write
\begin{align*}
 (\,  {}\!^t \mathcal{L})^{N+1}& = \,  {}\!^t \mathcal{L} (\,  {}\!^t \mathcal{L})^N  = (\nabla\cdot A ) (\,  {}\!^t \mathcal{L})^N = \sum_{\vert \alpha \vert \leq N} \sum_{i=1}^d \partial_i (A_i c_{\alpha,N}\partial^\alpha),\\
 &=  \sum_{\vert \alpha \vert \leq N}  (\text{div } A)c_{\alpha,N}\partial^\alpha  +\sum_{\vert \alpha \vert \leq N}   A\cdot \nabla c_{\alpha,N} \partial^\alpha + \sum_{\vert \alpha \vert \leq N} \sum_{i=1}^d  A_i c_{\alpha,N} \partial_i\partial^\alpha,\\
 &=  \sum_{\vert \gamma \vert \leq N+1} c_{\gamma,N+1}\partial^\gamma,
 \end{align*}
 where
 \begin{align*}
 &  \quad c_{0, N+1} = (\text{div } A) c_{0,N} +  A\cdot \nabla c_{0,N}, \quad \text{if} \quad \vert \gamma \vert =0,\\
& \quad c_{\gamma, N+1}= (\text{div } A)c_{\gamma,N} +A\cdot \nabla c_{\gamma,N} + A_i c_{\alpha,N}  \quad \vert \alpha \vert = \vert \gamma \vert -1,\text{ if } \quad 1\leq \vert \gamma \vert \leq N,\\
 &    \quad c_{\gamma,N+1} = A_i c_{\alpha,N}, \quad  \text{ if } \partial^\gamma = \partial_i \partial^\alpha, \vert \alpha \vert = N,\text{ if } \quad \vert \gamma \vert = N+1.
\end{align*}
We estimate now each coefficient. First of all $\partial^\beta c_{0,N+1}$ is a finite sum of terms of the form $(\partial^{\beta_1}\partial_i A_i)( \partial^{\beta_2} c_{0,N})$ and  $(\partial^{\beta_1}A_i)( \partial^{\beta_2} \partial_i c_{0,N})$ with $\beta = \beta_1 + \beta_2.$ Using \eqref{est:Ai} and the induction the first term  is bounded by 
$$\mathcal{F}(\mathcal{M}_{\vert \beta_1 \vert +2}) \sum_{k=2}^{\vert \beta_1\vert+2}\vert \nabla \Phi \vert^{-k} \mathcal{F}(\mathcal{M}_{ N + \vert \beta_2\vert +1}) \sum_{l=N}^{2N + \vert \beta_2\vert } \vert \nabla \Phi \vert^{-l}.  $$
Concerning the second term, if $\beta_1 = 0, \beta_2 = \beta$ it is bounded by 
$$\frac{1}{ \vert \nabla \phi \vert } \mathcal{F}(\mathcal{M}_{N+ \vert \beta \vert +2})\sum_{l=N}^{2N + \vert \beta  \vert +1} \vert \nabla \Phi \vert^{-l} \leq   \mathcal{F}(\mathcal{M}_{N+1+ \vert \beta \vert +1})\sum_{l=N+1}^{2(N+1) + \vert \beta  \vert} \vert \nabla \Phi \vert^{-l}. $$
If $\beta_1 \neq 0$ it is bounded by
    $$\mathcal{F}(\mathcal{M}_{\vert \beta_1 \vert +1}) \sum_{k=2}^{\vert \beta_1\vert+1}\vert \nabla \Phi \vert^{-k} \mathcal{F}(\mathcal{M}_{ N + \vert \beta_2\vert +1}) \sum_{l=N}^{2N + \vert \beta_2 \vert +1} \vert \nabla \Phi \vert^{-l}. $$
Since $ N+2 \leq k+l \leq 2N +2+ \vert \beta_1\vert + \vert \beta_2 \vert = 2(N+1) + \vert \beta \vert $ we see that $\partial^\beta c_{0,N+1}$ satisfies the estimate in \eqref{Ltranspose} with $N$ replaced by $N+1.$ 

Let us look to the term $\partial^\beta c_{\gamma,N+1}$ with $\vert \gamma \vert = N+1.$ This term is also a finite sum of terms of the form $(\partial^{\beta_1}A_i )(\partial^{\beta_2}c_{\alpha,N}), \vert \alpha \vert = \vert \gamma \vert -1.$  As above, if $\beta_1 =0,$ using \eqref{est:Ai} and the induction it is bounded by
\begin{align*}
 \frac{1}{ \vert \nabla \phi \vert } \mathcal{F}(\mathcal{M}_{N - \vert \gamma \vert +1+ \vert \beta \vert +1}) &\sum_{l=N}^{2N -\vert \gamma \vert +1+ \vert \beta  \vert  } \vert \nabla \Phi \vert^{-l}\\
 & \leq \mathcal{F}(\mathcal{M}_{N+1 - \vert \gamma \vert + \vert \beta \vert +1})\sum_{l=N+1}^{2(N+1) -\vert \gamma \vert + \vert \beta  \vert  } \vert \nabla \Phi \vert^{-l}.
 \end{align*}
 If $\beta_1 \neq 0$ it is   bounded by
$$\mathcal{F}(\mathcal{M}_{\vert \beta_1\vert +1}) \sum_{k=2}^{\vert \beta_1\vert+1}\vert \nabla \Phi \vert^{-k} \mathcal{F}(\mathcal{M}_{ N - \vert \gamma  \vert +1+ \vert \beta_2\vert}) \sum_{l=N}^{2N -\vert \gamma \vert+1+ \vert \beta_2 } \vert \nabla \Phi \vert^{-l}. $$
Since $N+2 \leq k+l \leq 2N +2- \vert \gamma \vert + \vert \beta \vert$ we see that $\partial^\beta c_{\gamma,N+1}$ satisfies also the estimate in \eqref{Ltranspose}. The estimates of the other terms are similar and left to the reader.
\end{proof}
\subsection{Proof of Theorem  \ref{PS2}}
Case 1. Let $\psi\in C_0^\infty(\xR)$ be such that $\psi(x) = 1$ if $\vert x \vert \leq 1$, $\psi(x) = 0$ if $\vert x\vert \geq 2.$ With the notation in \eqref{Ij0},   $j$ beeing fixed, we write
\begin{equation}\label{Ij}
\begin{aligned}
 I_j(\lambda)  = \int e^{i \lambda \Phi(\xi)}   \psi \big(\lambda^\mez &\vert \nabla \Phi(\xi)\vert\big) \chi_j(\xi)b(\xi)\, d\xi \\
 + &\int e^{i \lambda\Phi(\xi})  (1-\psi \big(\lambda^\mez \vert \nabla \Phi(\xi)\vert\big)) \chi_j(\xi)b(\xi)\, d\xi 
  =: K_j(\lambda) + L_j (\lambda).
 \end{aligned}
 \end{equation}
 We shall use (see \eqref{injective}) the fact that on the support of $\chi_j$ the map $\xi \mapsto \nabla \Phi(\xi)$ is injective. Let us estimate $K_j$. We write
 $$\vert K_j(\lambda) \vert \leq \int  \vert \psi \big(\lambda^\mez \vert \nabla \Phi(\xi)\vert\big) \chi_j(\xi) b(\xi) \vert \, d\xi$$
 and we set $\eta = \lambda^\mez \nabla \Phi(\xi)$ then $d\eta = \lambda^{\frac{d}{2}} \vert \text{det\,Hess }\Phi(\xi)\vert \, d\xi.$ Then using \eqref{hypE} $(iii)$ and   the notations therein we obtain
 \begin{equation}\label{est:Jj}
   \vert K_j(\lambda) \vert \leq \frac{C_d}{a_0} \mathcal{N}_0\lambda^{-\frac{d}{2}}.
   \end{equation}
 To estimate $L_j$ we introduce  the vector field  $X = \frac{1}{i \lambda}\frac{\nabla \Phi}{\vert \nabla \Phi\vert^2}\cdot \nabla $ which satisfies
 $$X e^{i \lambda \Phi} = e^{i \lambda \Phi}.$$
Now with $N \geq 1$ to be chosen we write
$$L_j(\lambda) = \int e^{i\lambda\Phi(\xi)} ({}^tX)^N\Big\{ (1-\psi \big(\lambda^\mez \vert \nabla \Phi(\xi)\vert\big)) \chi_j(\xi)b(\xi)\Big\}\, d\xi.$$
Since $X = \frac{1}{i\lambda}\mathcal{L}$ we can use \eqref{Ltranspose} and we obtain
\begin{align*}
  &\vert L_j(\lambda) \vert \leq C_N \sum_{\substack{\alpha= \alpha_1 + \alpha_2 + \alpha_3\\ \vert \alpha \vert \leq N}} S_{\alpha,N}, \quad \text{where}\\
     &S_{\alpha,N} =\lambda^{-N} \int \vert c_{\alpha,N} \vert \la\partial_\xi^{\alpha_1} \big[1-\psi \big(\lambda^\mez \vert \nabla \Phi(\xi)\vert\big)\big] \ra \la \partial_\xi^{\alpha_2}\chi_j(\xi)\ra   \la \partial_\xi^{\alpha_3}b(\xi)\ra \, d\xi. 
\end{align*}
Our aim is to prove that with an appropriate choice of $N$ we have
\begin{equation}\label{est:Kj}
\vert L_j(\lambda) \vert \leq  \mathcal{F}(\mathcal{M}_{d+2}) \mathcal{N}_{d+1}\frac{1}{a_0}\lambda^{-\frac{d}{2}}.
\end{equation}

Step 1. $\alpha_1 =0.$ Here we integrate on the set $\vert \nabla \Phi(\xi) \vert \geq \lambda^{-\mez}.$  We use \eqref{injective}, the bounds \eqref{derchij}, \eqref{Ltranspose}, \eqref{hypE} $(iii)$  and we make the change of variable $\eta = \nabla \Phi(\xi);$ then $d\eta =   \vert \text{det\,Hess }\Phi(\xi)\vert \, d\xi $ then
\begin{align*}
 \vert S_{\alpha,N} \vert &\leq \frac{\lambda^{-N}}{a_0}  \delta^{-\vert \alpha_2 \vert}Ê\mathcal{N}_{\vert \alpha_3\vert} \mathcal{F} (\mathcal{M}_{N+1}) \sum_{k=N}^{2N-\vert \alpha \vert} \int_{ \vert \eta \vert \geq \lambda^{-\mez} } \frac{d\eta}{\vert \eta \vert^k}\\
 &\leq \frac{C_d \lambda^{-N} }{a_0} \delta^{-\vert \alpha_2 \vert}Ê\mathcal{N}_{\vert \alpha_3\vert}  \mathcal{F} (\mathcal{M}_{N+1})\sum_{k=N}^{2N-\vert \alpha \vert}\int_{\lambda^{-\mez}}^{+ \infty} r^{d-1-k}\, dr.
\end{align*}
 Taking $N= d+1$ since $\vert \alpha_j  \vert \leq \vert \alpha  \vert \leq N$ we see that 
 \begin{align*} 
  \vert S_{\alpha,N} \vert &\leq   \frac{  \lambda^{-N}}{a_0}   \delta^{-\vert \alpha  \vert}Ê\mathcal{N}_{d+1} \mathcal{F} (\mathcal{M}_{d+2})\lambda^{N- \mez \vert \alpha  \vert -\frac{d}{2}}\\
  & \leq \frac{1 }{a_0}\mathcal{N}_{d+1} \mathcal{F} (\mathcal{M}_{d+2})\lambda^{- \frac{d}{2}} \frac{1}{(\lambda^\mez \delta)^{\vert \alpha \vert}}.
  \end{align*}
 Since by \eqref{delta} $\delta$ is proportional to $ a_0$ and by \eqref{reduc1} we have assumed that $\lambda^\mez a_0 \ge 1$ we obtain eventually $ \vert S_{\alpha,N} \vert \leq \frac{1 }{a_0}\mathcal{N}_{d+1} \mathcal{F} (\mathcal{M}_{d+2})\lambda^{- \frac{d}{2}} $
 
 Step 2. $\alpha_1 \neq 0.$ Here, since we differentiate $\psi$ we are integrating on the set  $ \lambda^{-\mez} \leq \vert \nabla \Phi(\xi) \vert \leq 2\lambda^{-\mez} \leq 1.$ We can therefore use \eqref{est:nablaphi}.
 
 We have to estimate
 $$S_{\alpha,N} = \lambda^{-N} \int \vert c_{\alpha,N} \vert \la \partial_\xi^{\alpha_1} \big[ \psi \big(\lambda^\mez \vert \nabla \Phi(\xi)\vert\big)\big] \ra \la \partial_\xi^{\alpha_2}\chi_j(\xi)\ra   \la \partial_\xi^{\alpha_3}b(\xi)\ra \, d\xi.$$
 By the Faa-di-Bruno formula we have
 $$ \partial_\xi^{\alpha_1} \big[ \psi \big(\lambda^\mez \vert \nabla \Phi(\xi)\vert\big)\big]  = \sum_{1 \leq \vert \beta \vert \leq \vert \alpha_1 \vert} a_{\alpha,\beta}\psi^{(\beta)} \big(\lambda^\mez \vert \nabla \Phi(\xi)\vert\big) \prod_{i=1}^s \big(\lambda^\mez \partial_\xi^{l_i} \vert \nabla \Phi \vert\big)^{k_i}$$
 where $a_{\alpha,\beta}$ are absolute constants, $\sum_{i=1}^s k_i = \beta, \sum_{i=1}^s k_i \vert l_i \vert = \vert \alpha_1 \vert.$ Using \eqref{est:nablaphi} we deduce that
 $$ \la \partial_\xi^{\alpha_1} \big[ \psi \big(\lambda^\mez \vert \nabla \Phi(\xi)\vert\big)\big] \ra \leq \mathcal{F}(\mathcal{M}_{\vert \alpha_1 \vert +1})\sum_{1 \leq \vert \beta \vert \leq \vert \alpha_1 \vert} \lambda^{ \frac{\vert \beta\vert}{2}} \la \psi^{(\beta)} \big(\lambda^\mez \vert \nabla \Phi(\xi)\vert\big) \ra \vert \nabla \Phi \vert^{\vert \beta \vert - \vert \alpha_1 \vert}.$$
Using \eqref{Ltranspose} and the fact that $\vert \nabla \Phi(\xi)\vert \leq 1$ we have $\vert c_{\alpha,N} \vert \leq \mathcal{F}(\mathcal{M}_{N +1}) \vert \nabla \Phi(\xi)\vert^{\vert \alpha \vert -2N}.$ Eventually we have $ \la \partial_\xi^{\alpha_2}\chi_j(\xi)\ra \leq C_\alpha \delta^{- \vert \alpha_2 \vert}.$ Performing as above the change of variables $\eta = \nabla \Phi(\xi)$ we will have
 \begin{align*}
   \vert S_{\alpha,N} \vert &\leq  \frac{\lambda^{-N}  \delta^{- \vert \alpha_2 \vert}}{a_0}\mathcal{F}(M_{N +1}) \mathcal{N}_{\vert \alpha \vert}\sum_{1 \leq \vert \beta \vert \leq \vert \alpha_1 \vert} \lambda^{\frac{\vert \beta\vert}{2}}\int_{\lambda^{-\mez} \leq \vert\eta \vert \leq 2\lambda^{-\mez}} \vert \eta \vert^{\vert \alpha \vert -2N + \vert \beta \vert - \vert \alpha_1 \vert}\, d\eta\\
  &\leq \frac{1}{a_0} \mathcal{F}(\mathcal{M}_{N +1})\mathcal{N}_{\vert \alpha \vert} \lambda^{-\frac{d}{2}} (\lambda^\mez \delta)^{-\vert \alpha_2 \vert}.
  \end{align*}
  Since $\lambda^\mez \delta \geq \mathcal{F}(\mathcal{M}_3)$ we obtain $\vert S \vert \leq \frac{1}{a_0}\mathcal{F}(\mathcal{M}_{N +1})\mathcal{N}_{N} \lambda^{-\frac{d}{2}}.$ Therefore \eqref{est:Kj} is proved. Using \eqref{Ij}, \eqref{est:Jj},  \eqref{est:Kj} since $N = d+1$ we obtain
  $$\vert I_j \vert \leq  \frac{1}{a_0} \mathcal{F}(\mathcal{M}_{d+2}) \mathcal{N}_{d+1}\lambda^{-\frac{d}{2}}, \quad 1 \leq j \leq J.$$
  Now since $J \leq C_d\delta^{-d} \leq \mathcal{F}(\mathcal{M}_3) a_0^{-d}$ using \eqref{Ij0} we obtain eventually
  $$\vert I(\lambda) \vert \leq \frac{1}{a_0^{1+d}} \mathcal{F}(\mathcal{M}_{d+2}) \mathcal{N}_{d+1}\lambda^{-\frac{d}{2}}$$
  which completes the proof of the first case of the theorem.
 
  We prove now the second part of Theorem \ref{PS2}.

In  that case it is not necessary to make a localization of $I(\lambda)$ in small balls of size $\delta$ as in the first case.

Then as before we write
\begin{equation}\label{I(lambda)}
\begin{aligned}
 I (\lambda)  = \int e^{i \lambda \Phi(\xi)}   \psi \big(\lambda^\mez &\vert \nabla \Phi(\xi)\vert\big)  b(\xi)\, d\xi \\
 + &\int e^{i \lambda\Phi(\xi)}   (1-\psi \big(\lambda^\mez \vert \nabla \Phi(\xi)\vert\big))  b(\xi)\, d\xi 
  =: K (\lambda) + L (\lambda) 
 \end{aligned}
 \end{equation}
and the final estimate follows from \eqref{est:Jj} and \eqref{est:Kj}.
 
 \end{document}